\documentclass[reqno,11pt]{amsart} 
\usepackage{amssymb,amsfonts,amsbsy,setspace,yfonts} 
\usepackage[all]{xy} 
\theoremstyle{plain}
\newtheorem{theorem}{Theorem}[section]
\newtheorem{definition}[theorem]{Definition}
\newtheorem{lemma}[theorem]{Lemma}

\numberwithin{equation}{section} %% Comment out for sequentially-numbered
\numberwithin{figure}{section} %% Comment out for sequentially-numbered
\newtheorem{prop}[theorem]{Proposition}
\theoremstyle{remark}
\newtheorem{remark}[theorem]{Remark}

\newtheorem{example}[theorem]{Example}
\newtheorem*{acknowledgement*}{Acknowledgement}

\def\R{{\mathbb R}} 
\def\N{{\mathcal N}} 
\def\W{{\mathcal W}}

\def\M{{\mathcal M}} 
\def\C{{\mathcal C}} 
\def\H{{\mathcal H}} 
\def\L{{\mathcal L}} 

\def\D{{\mathcal D}}

\def\T{{\mathcal T}}
\def\clomega{\Omega}
\def\res{\mathop{\hbox{\vrule height 7pt width.5pt depth 0pt\vrule height .5pt width 6pt depth 0pt}}\nolimits}
\def\div{\mathop {\rm div}\nolimits}
\def\spt{\mathop {\rm spt}\nolimits}
\def\eps{\varepsilon}
\def\tang{{\textgoth T}}
\def\norm{{\textgoth N}}

\def\lip{Lip(\Omega)}
\def\lipp{{\bf X}_0(\Omega)}
\def\lips{{\bf X}_0^\sharp (\Omega)}

\def\xsp{{\bf X}(\Omega)}
\def\xn{{\bf X}_0(\Omega)}
\def\x0s{{\bf X}_0^\sharp (\Omega)}

\def\ds{\displaystyle}

\title[The Monge-Kantorovich problem for distributions]{The
  Monge-Kantorovich problem for distributions and applications}
\author{Guy Bouchitt\'e, Giuseppe Buttazzo, Luigi De Pascale}

\address{Laboratoire d' Analyse Non Lin\'eaire Appliqu\'ee, U.F.R. des Sciences et Techniques, Universit\'e du Sud Toulon-Var, Avenue de l'Universit\'e, BP 20132, 83957 La Garde Cedex, France}
\email{bouchitte@univ-tln.fr}

\address{Dipartimento di Matematica, Universit\`a di Pisa, Largo Pontecorvo 5, 56127 Pisa, Italy}
\email{buttazzo@dm.unipi.it}

\address{Dipartimento di Matematica Applicata, Universit\`a di Pisa, Via Buonarroti 1/C, 56127 Pisa, Italy}
\email{depascal@dm.unipi.it}

\begin{document}

\begin{abstract}
We study the Kantorovich-Rubinstein transhipment problem when the
difference between the source and the target is not anymore a balanced measure but belongs to a suitable subspace $\xsp$ of first order distribution. A particular subclass $\x0s$ of such distributions
will be considered which includes the infinite sums of dipoles
$\sum_k(\delta_{p_k}-\delta_{n_k})$ studied in \cite{P1, P2}.
In spite of this weakened regularity, it is shown that an optimal transport density still exists among  nonnegative finite measures. Some geometric properties of the Banach spaces 
$\xsp$ and $\x0s$ can be then deduced.
\end{abstract}
\maketitle

\begin{flushleft}

\bigskip

\end{flushleft}

%%%%%%%%%%%%%%%%%%%%%%%%%%%%%%%%%%%%%%%%%%%%%%%%%%%%%%%%%%%%%%%%%%%%%%%%%%%%%
%%%%%%%%%%%%%%%%%%%%%%%%%%%%%%%%%%%%%%%%%%%%%%%%%%%%%%%%%%%%%%%%%%%%%%%%%%%%%
\bigskip
\hfill{\it Dedicated to Hedy Attouch}\par
\hfill{\it on the occasion of his 60th birthday.}
\bigskip

\section{Introduction}
In the recent years, motivated by many applications, a lot of
attention has been devoted by the mathematical community to mass
transportation problems. They can be expressed in different equivalent
formulations, that we will shortly recall below. The usual
setting for these problems requires to consider source and target in
the space of probability measures on a domain of $\R^N$, on a
manifold, or more generally on a metric space. On the other hand, for
various applications (see for instance \cite{Br, P1, P2}) it is
interesting to develop a theory of optimal transportation (and
Wasserstein distances) for more general objects. In this paper
$\Omega$ will be a convex compact subset of $\R^N$ and we will focus
our attention on the space of distributions of order one and zero ``average''
\begin{multline}\label{xn}
\xn := \{f \in {\mathcal D}' (\R^N) \ | \ \forall \varphi \in
{\mathcal D} (\R^N), \ \langle f, \varphi \rangle \leq C (\|\varphi
\|_{L^\infty (\Omega)} +\\\|\nabla \varphi \|_{L^\infty (\Omega)}), \
\langle f , 1 \rangle =0 \}. 
\end{multline}
Such distributions are compactly supported in $\Omega$
and the condition $\langle f , 1 \rangle =0$ above means that whenever
$\varphi \in {\mathcal D} (\R^N)$ is constant in $\Omega$ then   
$\langle f , \varphi \rangle =0$.  From (\ref{xn}),  
it is natural to endow $\xn$ with the dual of the Lipschitz norm 
on smooth functions and we may identify $\xn$
with a closed subspace of the dual of $\C^1(\Omega)$.
Let us notice that , although  it is tempting, we are not allowed to  identify  
$ \xn$ with a subspace of the dual of $Lip (\Omega)$ since the extension of an 
element $f\in \xn$ to all Lipschitz functions given by Hahn-Banach Theorem 
{\em  is non unique ! }.

As far as  the usual setting for the Monge-Kantorovich problem is considered,
one needs to work only with the subspace  of measures of $\xn$ given by
 $$ \M _0(\Omega):=\left\{f \in \M(\Omega) \ | \ \int_\Omega f =0\right\}.$$
It is shown in \cite{B-C-J} that the closure $\x0s$ of $\M_0(\Omega)$ can be characterized as 
\begin{multline*} 
\x0s:=\{f \in\xn \ | \ \forall \eps>0\ \exists \ C_\eps>0 \ s.t.\\
  |\langle f,\varphi \rangle|\le C_\eps\|\varphi\|_{L^\infty (\Omega)}+
\eps\|\nabla \varphi\|_{L^\infty (\Omega)} \ \forall  \varphi \in {\mathcal D}(\R^N) \}\ .
\end{multline*}
This strict subspace can  can be seen also as the completion 
of $\M _0(\Omega)$ with respect to the Monge-Kantorovitch norm.
In \cite{B-C-J} it is proved that elements of $\x0s$ can be represented 
as the distributional divergence of functions in $L^1(\Omega;\R^N)$ (or more in general 
of  a suitable class  of tangential vector measures).
 The role of this space will stand out for different reasons which will be
clear through the paper. In particular by suitably extending the
Monge-Kantorovich problem to all $\xn$, we will construct a 
 linear continuous projector from $\xn$ onto $\x0s$ (see theorem \ref{projector}).

\medskip
Before considering the details of the weakened formulation of the
Monge-Kanto-\\rovich mass transportation problem let us first recall the
main issues about the classical version of the problem and its various
formulations. Also in this short survey, for simplicity, we will limit ourselves to the case when
$\Omega$ is a  convex compact subset of $\R^N$.

\medskip

$\bullet$ The most classical formulation of a mass transportation
problem goes back to Monge (1781) and in modern terminology
(Kantorovich 1942) consists, given two probabilities $f^+$ and $f^-$
on $\Omega$, in finding a measure $\gamma$ on $\Omega\times\Omega$
which minimizes the total transportation cost
\begin{equation}\label{formulazione1}
  \int_{\Omega\times\Omega}|x-y|\,d\gamma(x,y)
\end{equation}
among all admissible transport plans $\gamma$ such that
$\pi^1_\sharp\gamma=f^+$ and $\pi^2_\sharp\gamma=f^-$. Here
$\pi^1_\sharp$ and $\pi^2_\sharp$ are the usual push-forward operators
associated to the projections $\pi^1$ and $\pi^2$ from
$\Omega\times\Omega$ on the first and second factors respectively.
Notice that the formulation above is meaningless if $f^+$ and $f^-$
are distributions that are not measures.

We say that the problem is of {\it transhipment} type (see Kantorovich
and Rubinstein \cite{CR}) when only the difference $f^+-f^-$ is
specified.

\begin{definition}\label{norm}
  The quantity
  \begin{equation}
    \W^1 (f^+,f^-):=\inf_{\pi^1_\sharp\gamma=f^+,\ \pi^2_\sharp\gamma=f^-}\int_{\Omega\times\Omega}|x-y|\,d\gamma(x,y)
  \end{equation}
  is called Wasserstein distance of $f^+$, $f^-$. Setting
  $f=f^+-f^-$ the quantity above may be redefined as:
  \begin{equation}
    \inf_{\pi^1_\sharp\gamma-\pi^2_\sharp\gamma=f}\int_{\Omega\times\Omega}|x-y|\,d\gamma(x,y)
  \end{equation}
  and we will denote it by $ \W^1(f)$. This last quantity is a norm in
  the space of measures $f$ such that $\int df=0$ and it is called the
  Kantorovich norm.
\end{definition}

$\bullet$ The dual formulation of the mass transportation problem
(\ref{formulazione1}) introduces the Kantorovich potential $u$ which
is a solution of the maximization problem
\begin{equation}\label{formulazione3}
  \max\Big\{\int u\,d(f^+-f^-)\ :\ u\in Lip_1(\Omega)\Big\}=\W^1 (f).
\end{equation}
The value of (\ref{formulazione3}) is called flat dual norm
$\|f^+-f^-\|$ in the duality $\langle\lip,Lip'(\Omega)\rangle$. In the
classical setting $u$ plays a key role in proving many of the results
of the Monge-Kantorovich theory.

\medskip

$\bullet$ The mass transportation problem above can be equivalently
expressed through the {\it Kantorovich potential} $u$ and the
transport density $\mu$ which solve the system (see \cite{B-B, EG,
  A1})
\begin{equation}\label{formulazione2}
  \left\lbrace\begin{array}{ll}
      -\div(\mu\,D_\mu u)=f^+-f^- & \ \ \mbox{in} \ \Omega \\
      |D_\mu u|=1\ \mbox{ on }\spt\mu,& \ \ u\mbox{ is 1-Lipschitz on }\Omega
    \end{array}\right.
\end{equation}
which consists of an elliptic PDE coupled with an eikonal
equation. For general measures $\mu$ the precise sense of the PDE
above has to be intended by means of a weak formulation involving the
theory of Sobolev spaces with respect to a measure (see
\cite{B-B-S,B-B}). The following representation formula for
$\mu$ holds: 
$$\mu=\int\H ^1\res S_{x,y}\, d\gamma(x,y)$$
where $\gamma$ is an optimal transport plan for the cost (\ref{formulazione1}), $\H^1$ is
the 1-dimensional Hausdorff measure and $S_{x,y}$ is the geodesic line
(the segment in our Euclidean case) connecting $x$ to $y$. The
transport density $\mu$ appears in various applications whose models admit a
Monge-Kantorovich type formulation (see for example \cite{B-B, B-B-S,
  C-C}). Moreover the transport density was also used to prove an
existence result for optimal transport maps (see \cite{EG}).

\medskip

\noindent Several results on the formulations above have been
obtained; in particular, it has been shown that the regularity of the
measure $\mu$ (called transport density) depends on the regularity of
the data $f^+$ and $f^-$. More precisely, we summarize here below what
is known on this dependence.

\medskip

$\bullet$ When $f^+$ and $f^-$ are merely nonnegative measures, the
transport density $\mu$ is a nonnegative measure too (see \cite{B-B}).
As already mentioned above, the Monge-Kantorovich PDE
(\ref{formulazione2}) has to be intended in the sense of Sobolev
spaces with respect to a measure.

\medskip

$\bullet$ Additional assumptions on the source terms $f^+$ and $f^-$
have to be made in order to provide more regularity to the transport
density $\mu$; more precisely, if $f^+$ and $f^-$ are in $L^p(\Omega)
$ with $1\le p\le+\infty$ then $\mu$ is in $L^p(\Omega)$ too (see
\cite{DP-P}, \cite{DP-E-P}). Moreover, in these cases the transport
density $\mu$ turns out to be unique (see \cite{A1}).

\medskip

$\bullet$ Some recent results on the particular case of mass
transportation problems that intervene in the identification of sand
pile shapes (see \cite{C-C}, \cite{Gtesi}) indicate that the H\"older
continuity of $\mu$ has to be expected, under additional regularity on
the data, while simple examples show that nothing more than Lipschitz
regularity can be obtained for $\mu$ even if very strong regularity
hypotheses on the data are made.

\medskip

$\bullet$ The continuity of $\mu$, when $f$ is continuous, as been obtained in
\cite{FraGelPra} under some additional geometric assumptions on
the supports of $f^+$ and $f^-$. However the continuity of $\mu$ in the general case
is an open problem.

\medskip

In the present paper we will also consider the opposite question
of the existence of a transport density $\mu$ when the source datum
$f=f^+-f^-$ is less regular than a measure. Indeed, we assume that $f$
only belongs to the space $\lipp$.

As it is well explained in \cite{brezis2,P1,P2} in some applications
$f$ describes the location and the topological degree of
singularities of a map $u$ with values in the sphere.

Indeed if $u$ belongs to $W^{1,N-1}(\Omega,\R^N)$ and is bounded
($u\in W^{1,N-1 }(\Omega,\mathcal S^{N-1})$ is a particular case) then
by the $*$-Hodge operator we can say that the $N-1$ form $$D(u):=
\sum_{i=1}^N\frac{(-1)^{i-1}}{N}u_i\widehat{du}^i, \ \ \
(\mbox{where}\
\widehat{du}^i:=du_i\wedge\dots\wedge\widehat{du_i}\wedge\dots)$$
corresponds to an $L^1(\Omega,\R^N)$ vector field, and then
$$-\div(*D(u))\in\D'(\Omega).$$
More precisely $-\div(*D(u))$ belongs
to $\lipp$.  For a smooth map $u$ $$-\div(*D(u))=*dD(u)=*J(u)$$
then
with a little abuse of words the Jacobian of a bounded map in
$W^{1,N-1}(\Omega,\R^N)$ belongs to $\lipp$ (actually to the smaller
subspace $\x0s$ as we will see).

If $J(u)$ is a measure then the Kantorovich norm of $J(u)$ correspond
to the mass of the minimal connection of $J(u)$ which plays a key role
in the relaxation of the Dirichlet functional (see \cite{brezis2}) as
well as in the Ginzburg-Landau theory \cite{brezis2, sandier}.

Another interesting issue is to establish weather an $N$-form (or
equivalently a distribution) is a distributional Jacobian or not (see
for example \cite{D-M, R-Y}). A related question is to establish when
a distribution in $\lipp$ can be approximated weakly by Jacobians and
in the negative case one may try to give a quantitative answer. We
will study this question in the last section of the paper.

The next example introduces a relevant class of distributions in
$\lipp$ which appear as distributional Jacobians in the theory of the
Ginzburg-Landau equations and which has been studied in
\cite{B-B-M,P1,P2}.

\begin{example}
  Given two sequences of points $\{p_i\},\ \{n_i\} $ in $\Omega$ such
  that $\sum_{i=0}^\infty|p_i-n_i|<\infty$ we consider the
  distribution $$\langle T,u\rangle:=\sum_{i=1}^\infty
  u(p_i)-u(n_i)\qquad\forall u\in\lip.$$
  It is easy to see that
  $T\in\lipp$; let us show that $T$ is also in the space $\lips$. 
  Let $\eps>0 $ and consider $k$ such
  that $\sum_{i>k}|p_i-n_i|\le\eps$; then $$|\langle
  T,u\rangle|\le\sum_{i\le k}|u(p_i)-u(n_i)|
  +\sum_{k<i}|u(p_i)-u(n_i)|\le 2k\|u\|_\infty+\eps Lip(u).$$
  Notice
  that in this case it is not possible to define a positive and a
  negative part of $f$.
\end{example}

The plan of the paper is the following:
we will extend the Kantorovich norm and the mass
transportation problem to the space of distributions $f\in \lipp$.
We show that for a wide class of
source data (namely for $f\in\lips$) the transport density $\mu$ still
remains a measure. We will then show that the space $\xn$ may be
decomposed in the direct sum of $\x0s$ and of the space of divergences
of normal measures. This decomposition is ``orthogonal'' in the sense
of the Wasserstein norm.
Some of the ideas are connected with the papers
\cite{P1,P2}, \cite{B-C-J} and \cite{O}, where very interesting tools
for studying the distributions of $\lipp$ have been introduced. In
particular our Theorem \ref{exist&dual} extends (in the most natural
reformulation) Theorem 1 of \cite{O} to the case $k=1$ of functions in
$\C^{0,1}(\Omega)$.

%%%%%%%%%%%%%%%%%%%%%%%%%%%%%%%%%%%%%%%%%%%%%%%%%%%%%%%%%%%%%%%%%%%%%%%%%%%%%%
%%%%%%%%%%%%%%%%%%%%%%%%%%%%%%%%%%%%%%%%%%%%%%%%%%%%%%%%%%%%%%%%%%%%%%%%%%%%%%
\section{Preliminary results of Functional Analysis and Measure Theory}
\subsection{Completion of dual spaces} 

It looks nice to formulate the question of existence of a Kantorovich potential 
(see (\ref{formulazione3})) in an abstract setting. We are then considering two 
separable normed spaces $X$ and $Y$, with $Y\hookrightarrow X$ (in our context 
$X$ will be $\C(\Omega) $ and $Y=\lip$). We assume that the
injection above is dense (i.e. $Y$ is dense in $X$ with the norm of
$X$) and compact (i.e. bounded sequences in $Y$ are relatively compact
in $X$). Moreover we assume that the norm of $Y$ is l.s.c. with
respect to the convergence in $X$.

 It is well known as James's theorem
(see for instance Remark 3 in Chapter 1 of \cite{brezis}) that $Y$ is
reflexive if and only if the supremum in the dual norm of $Y'$
$$\|f\|_{Y'}:=\sup\{\langle f, u\rangle\ :\ u\in Y,\ \|u\|_Y\le1\}$$
is attained for every $f\in Y'$. In the situation which is interesting
for our purposes $Y$ is not reflexive and we are going to consider the
elements $f$ of a space that is smaller than the whole dual space
$Y'$. More precisely, we denote by $Y^\#$ the space of all $f\in Y'$
such that for every $\eps>0$ there exists $C_\eps>0$ which verifies
$$|\langle f,u\rangle|\le C_\eps\|u\|_X + \eps\|u\|_Y\qquad\forall
u\in Y.$$

We endow $Y^\#$ with the norm of $Y'$. The obvious inclusions are
that $X'\subset Y^\#\subset Y'$. In addition to the conditions above
on $X$ and $Y$ we assume that there exists a family
$T_\delta\in\L(X,Y)$ such that
\begin{equation}\label{ipot1}
  \lim_{\delta \to 0} \|T_\delta u-u \|_X=0 \ \ \mbox{for every}\ u\in Y;
\end{equation}
\begin{equation}\label{ipot2}
  T_\delta\ \mbox{is bounded in}\ \L(X,Y).\ (\mbox{i.e.}\ \|T_\delta u\|_Y
  \le C\|u\|_X).
\end{equation}
When $X$ and $Y$ are Hilbert spaces it is enough to take
$T_\delta=j^t$ for all $\delta$, where $j^t$ is the transpose of the
injection map $j:Y\hookrightarrow X$. When $X=\C_b(\Omega)$ or $\C(\Omega)$ and
$Y=\lip$ it is enough to take $T_\delta u=\rho_\delta*u$ where
$\rho_\delta$ is a family of smooth convolution kernels converging
weakly* to $\delta_0$.

\begin{prop} 
  The space $Y^\#$ coincides with the closure of $X'$ in the dual
  space $Y'$.
\end{prop}

\begin{proof}
  If $\{f_n\}$ is a sequence in $X'$ which converges to $f$ in the
  dual space $Y'$ we have $$|\langle f,u\rangle|\le|\langle
  f_n,u\rangle|+|\langle f-f_n,u\rangle| \le C_n\|u\|_X+\eps\|u\|_Y$$
  where we denoted $$C_n=\|f_n\|_{X'},\qquad\eps=\|f-f_n\|_{Y'}.$$
  Therefore $f\in Y^\#$. Vice versa, if $f\in Y^\#$, take
  $f_\delta=f\circ T_\delta$; in order to show that $f$ is in the
  closure of $X'$ in the dual space $Y'$ it is enough to show that
  $$\lim_{\delta\to0}\langle f_\delta,u\rangle=\langle
  f,u\rangle\qquad\forall u\in Y.$$
  Fix $u\in Y$; by the definition of
  $Y^\#$ we have for every $\eps>0$ $$|\langle
  f-f_\delta,u\rangle|=|\langle f,u-T_\delta u\rangle| \le
  C_\eps\|u-T_\delta u\|_X+\eps\|u-T_\delta u\|_Y.$$
  By the
  assumptions (\ref{ipot1}) and (\ref{ipot2}), passing to the limit as
  $\delta\to0$ in the inequality above gives
  $$\limsup_{\delta\to0}|\langle
  f-f_\delta,u\rangle|\le\eps\|u\|_Y(1+C)$$
  which implies our claim
  since $\eps>o$ was arbitrary.
\end{proof}

\begin{prop}\label{attained}
  For every $f \in Y^\#$ the supremum in the dual norm $$\sup\{\langle
  f,u\rangle\ :\ u\in Y,\ \|u\|_Y\le1\}$$
  is attained.
\end{prop}

\begin{proof}
  Let $\{u_n\}$ be a maximizing sequence for the dual norm above;
  since the injection $Y\hookrightarrow X$ is compact, we may assume
  that $u_n\to u$ in $X$ for some $u\in Y$ with $\|u\|_Y\le1$.
  Therefore by using the fact that $f\in Y^\#$, $$|\langle
  f,u_n\rangle-\langle f,u\rangle|=|\langle f,u_n-u\rangle| \le
  C_\eps\|u_n-u\|_X+2 \eps.$$
  Passing now to the limit as $n\to\infty$
  and using the fact that $\eps>0$ is arbitrary, we obtain that
  $$\lim_{n\to\infty}\langle f,u_n\rangle=\langle f,u\rangle$$
  which
  shows that $u$ is a maximizer for the dual norm.
\end{proof}

There is another characterization of the elements of $Y^\#$.
\begin{prop}\label{character2}
  We have $f\in Y^\#$ if and only if $\langle f,u_n\rangle\to0$ for
  every $u_n\to 0$ in $X$ with $\|u_n\|_Y$ bounded.
\end{prop}

\begin{proof}
  Take $f\in Y'$ and $u_n\to0$ in $X$ with $\|u_n\|_Y$ bounded. Then
  $$|\langle f,u_n\rangle|\le C_\eps\|u_n\|_X+\eps\|u_n\|_Y$$
  which
  gives, at the limit as $n\to\infty$, $\langle f,u_n\rangle\to0$.

  Vice versa assume by contradiction that there exist $\eps_0>0$ and a
  sequence $\{u_n\}$ such that
  \begin{equation}\label{contrad}
    \langle f,u_n\rangle\ge n\|u_n\|_X+\eps_0\|u_n\|_Y.
  \end{equation}
  With no loss of generality we can suppose $\|u_n\|_Y=1$ so that
  $$\|u_n\|_X\le\frac{1}{n}(\langle
  f,u_n\rangle-\eps_0)\le\frac{K}{n}.$$
  Then $u_n\to0$ in $X$, which
  implies $\langle f,u_n\rangle\to0$ by the hypothesis. This is a
  contradiction with (\ref{contrad}), which gives $\langle
  f,u_n\rangle\ge\eps_0$.
\end{proof}

\subsection{Some tangential calculus for measures}
Let $\mu$ be a Radon measure in $\R^N$. Following  \cite{CV,rev,BF}, 
we introduce the tangent space $\T_{\mu}$ to the measure $\mu$ which is 
  defined $\mu$ a.e. by setting  $\T_\mu(x):=\N_\mu^\bot(x)$ where
(see \cite{B-C-J} for further details related to the $L^\infty$-case under consideration here):
\begin{eqnarray*}
\N_\mu(x)&=& \{ \xi(x) \ :\ \xi \in \N_\mu \} \qquad \text{being} \\
\N_\mu &=&\{\xi\in (L^\infty_\mu)^N \ :\ \exists u_n\to0,\ u_n\ \mbox{smooth},\ \nabla
u_n \rightharpoonup\xi\ \mbox{weakly* in }L^\infty_\mu\}
\end{eqnarray*}
It turns out that the subspaces $\T_\mu$ and $\N_\mu$) are local in the sense that 
$\xi\in \T_\mu$ (resp. $\N_\mu$) iff  $\xi(x)\in \T_\mu(x)$
(resp. $\xi(x)\in \N_\mu(x)$) holds $\mu$-a.e.

We may now define an intrinsic notion of tangential and normal vector measures
in $\M^N(\Omega)$. It will be useful in the construction of a complement
$\x0s$ in the space
 $\xn$.

\begin{definition}\label{tangmeas} Let $\lambda\in \M(\Omega)^N$.
If $\lambda$ can be decomposed as $\lambda=v\, \mu$ where $\mu$ is a
positive Radon measure and 
$v \in (L^1_\mu)^N$ satisfies  $v(x)\in\T_\mu(x)$ $\mu$-a.e., 
then we say that $\lambda$ is a
  tangential measure. If alternatively $v(x)\in\N_\mu(x)$ $\mu$-a.e.
  we  say that $\lambda$ is a normal measure. 
 We will denote by $\tang$ the space of tangential
  measures and by
  $\norm$ the space of normal measures.
  \end{definition}
  Clearly we have the decomposition in direct sum 
  $$(\M(\Omega))^N\ =\ \tang \oplus \norm\ .$$

The following is a basic and intuitive lemma on tangent spaces which
will be used in the last section of the paper.

\begin{lemma}\label{submeas}
Let $\alpha$ and $\mu$ be two nonnegative Radon measures in $\R^N$ such that 
$\mu= \alpha +\mu_s$ where $\mu_s$ is singular with respect to $\alpha$. Then
$$ T_\mu(x)\subset T_\alpha (x) \ \ \ \alpha-a.e..$$ 
\end{lemma}

\begin{proof} We will prove that $\N_\alpha \subset \N_\mu$ and then the thesis 
will follow from the definition of tangent space. Since $\mu= \alpha +\mu_s$ where $\mu_s$ is
  singular with respect to $\alpha$, if $g \in (L^1_\mu)^N$ then $g \in (L^1_\alpha)^N$.
 Consider $\xi \in \N_\alpha$, then $\xi$ is the weak-* limit in $(L^\infty _\alpha)^N$ 
of a sequence $\{\nabla \varphi_n\}$ with $|\varphi_n| \leq c$ and $\varphi_n \to 0$ uniformly.
The sequence  $\{\nabla \varphi_n\}$ is bounded in $(L^\infty_\mu)^N$ and up to subsequences  
converges to $\tilde{\xi}$ weakly* in $(L^\infty_\mu)^N$.  Let $g \in (L^1_\mu)^N$; then
on the one hand 
$$\int \nabla \varphi _n \cdot g d \alpha \to \int \xi \cdot g d \alpha  $$
on the other hand
 $$\int \nabla \varphi _n \cdot g d \mu \to \int \tilde{\xi} \cdot g d \mu.$$
The last equality reads as
 $$\int \nabla \varphi _n \cdot g d \alpha+ \int \nabla \varphi _n \cdot g d \mu_s 
\to \int \tilde{\xi} \cdot g d \alpha+ \int \tilde{\xi} \cdot g d \mu_s, $$
then $\tilde{\xi}=\xi$ $\alpha-$a.e. and the conclusion follows.

\end{proof}

%%%%%%%%%%%%%%%%%%%%%%%%%%%%%%%%%%%%%%%%%%%%%%%%%%%%%%%%%%%%%%%%%%%%%%%%%%%%%
%%%%%%%%%%%%%%%%%%%%%%%%%%%%%%%%%%%%%%%%%%%%%%%%%%%%%%%%%%%%%%%%%%%%%%%%%%%%%
\section{The Optimal Transport Problem}

We will now  extend the different formulations of the optimal transport
problem to a distribution $f\in\lipp$, we will compare these formulations
to the classical case of measures and then  prove the main existence
theorems for minimizers.

\subsection{Kantorovich potential and optimal transport measure} 

In this subsection we will see that the classical theory can be easily extended
provided the distribution $f$ belongs to the subspace $\lips$.
To that aim we simply particularize the results of the previous section in the case $X=\C(\Omega) $ and $Y=\lip$ endowed with their natural norms. 
Then $Y^\sharp$ coincides with $\lips$ and by Proposition \ref{attained} if
$f\in\x0s$ then
\begin{equation}\label{newflat}
  \sup\{\langle f,u\rangle\ :\ u\in\lip,\ \|u\|_{\lip}\le1\}
\end{equation}
is attained. If $f$ has ``zero average'', that is $\langle
f,1\rangle=0$, we may replace the constraint $\|u\|_{\lip}\le1$ by the
seminorm inequality $\|\nabla u\|_{L^\infty(\Omega)}\le1$ thus
obtaining the flat norm of $f$.  The subspace of the restrictions to
$\Omega$ of functions in   $\C^1_0(\R^N)$ is dense
in $Lip (\Omega)$ for the $\C^0$ topology. Then if $f
\in\lips$ by Proposition \ref{character2} $$\max_{\varphi\in
  Lip_1}\langle f,\varphi\rangle =\sup_{\C^1_0(\R^N)\cap Lip_1}\langle
f,\varphi\rangle.$$
\begin{definition}\label{Was1} For every $f\in \xn$ the
  Wasserstein norm of $f$ is defined by
  \begin{equation}\label{was1}
    \W^1(f):=\sup\{\langle f,u\rangle\ :\ u\in\ \C^1_0(\R^N),\ \|\nabla u\|_\infty\le1\}.
  \end{equation}
\end{definition}
By Proposition \ref{character2} if $f\in\lips$ the sup in (\ref{was1})
does not change if performed on $Lip_1(\Omega)$ instead than
$\C^1_0(\R^N)\cap Lip_1$ and it is attained on $Lip_1(\Omega)$.  
We notice however that the supremum is not achieved in general for
$f\in \xn\subset\lips$. It is therefore worth to characterize those elements 
$f$ which belong to $\lips$.
Two such characterizations appeared in \cite{B-C-J}; we report here the
first one while the second will be given later in this
section.

\begin{theorem}[\cite{B-C-J}]\label{divelle1}
  Let $f\in\lipp$, then $f\in\lips$ if and only if there exists a
  vector field $\nu$ in $L^1(\Omega,\R^N)$ such that $-\div\nu=f$.
\end{theorem}

Let us now introduce a duality argument: if we define the mapping
$$h(p)=-\sup\{\langle f,u\rangle\ :\ u\in\lip,\ \|\nabla
u+p\|_{L^\infty(\Omega)}\le1\}$$
for every continuous function $p$, we
have that the Fenchel transform defined for all vector measure
$\lambda$ by
$$h^*(\lambda)=\sup_p\langle\lambda,p\rangle-h(p)=\sup_{p,u}
\{\langle\lambda,p\rangle+\langle f,u\rangle\ :\ \|\nabla
u+p\|_{L^\infty(\Omega)}\le1\},$$
is given by $$h^*(\lambda)=\left\{
  \begin{array}{ll}
    \int|\lambda| &\mbox{if}\ -\div\lambda=f\ \mbox{in}\ \D'\\
    +\infty &\mbox{otherwise}.
  \end{array}
\right.$$
The duality relation $\inf_\lambda h^*(\lambda)=-h(0)$ then
reads
\begin{multline}\label{min2}
  \max\{\langle f,u \rangle \ : \ \|\nabla u\|_{L^\infty (\Omega)} \le1\}\\
    =\min\{\int|\lambda|\ :\ \lambda\in\M^n(\Omega),\ -\div\lambda=f\ \in\ \D'\}.
 \end{multline}
The existence and the structure of an optimal $\lambda$ for the right
hand side of (\ref{min2}) has been discussed in \cite{B-C-J} for 
$f\in \lips$ and will be analyzed for $f \in \xn$ in the
next section.  We summarize as follows.
\begin{theorem}For every $f \in \lips$
  \begin{enumerate}
  \item there exists a Kantorovich potential $u$ which maximizes the
    quantity $$\W^1(f)=\max \{\langle f,u \rangle \ : \ \|\nabla
    u\|_{L^\infty (\Omega)} \le1\};$$
  \item there exists an optimal measure $\lambda$ which solves the
    problem
    \begin{equation}\label{divcurr}
      \min \{\|\lambda\|\ :\ -\div\lambda=f\ \in\ \D'\};
    \end{equation}
  \item the two extremal values in (\ref{newflat}) and (\ref{divcurr})
    are equal; i.e.
    \begin{equation}
      \W^1(f)=\min\{\|\lambda\|\ :\ -\div\lambda=f\ \in\ \D'\}.
    \end{equation}
  \end{enumerate}
\end{theorem}

Note that, being $f\in\lips$, the equality $-\div\lambda=f$ can be
equivalently considered either in $\D'$ or in the duality $\langle
Lip'(\Omega),\lip\rangle$. If $\mu$ denotes the total variation
$|\lambda|$ of $\lambda$, we can write $\lambda=v\mu$ for a
suitable vector field $v \in(L^1_\mu(\Omega))^N$. The measure
$\mu$ is the transport density which appear in the Monge-Kantorovich
PDE (\ref{formulazione2}) and we remark that $\mu$ is still a measure
even if $f$ is only in $\lips$.

The next proposition  (proved in \cite{B-C-J}) characterizes the distributions in $\lips$
as the divergence of tangential measures in $\tang$. 
We report here the proof of one of the two implications.

\begin{prop}\label{divtang}
  If $v \in(L^1_\mu)^N$ is such that $\div(v \mu)\in\lips$,
  then we have $v(x)\in\T_\mu(x)$ for $\mu$-a.e. $x$. Vice versa
  if $v (x)\in\T_\mu(x)$ $\mu$-a.e. then
  $\div(v \mu)\in\lips$.
\end{prop}

\begin{proof} We prove only the first of the two implications, for the
  complete proof we refer to \cite{B-C-J}. If $\xi$ is a normal vector
  field to $\mu$, that is $\xi(x)\in\N_\mu(x)$ for $\mu$-a.e. $x$, we
  have, denoting by $\{u_n\}$ the sequence corresponding to $\xi$
  $$\int\xi\cdot v\,d\mu=\lim_{n\to\infty}\int\nabla
  u_n\cdot v\,d\mu= \lim_{n\to\infty}-\langle
  u_n,\div(v \mu)\rangle.$$
  This last limit vanishes because
  $\div(v \mu)\in\lips$ by hypothesis, and $u_n\to0$ uniformly
  with $\|\nabla u_n\|_{L^\infty(\Omega)}$ bounded (see Proposition
  \ref{character2}). Therefore $v(x)\in\T_\mu(x)$ for $\mu$-a.e.
  $x$.
\end{proof}

The analysis performed in \cite{B-B} can be made also in the more
general case of $f \in\lips$; it is enough to repeat step by step what
done in \cite{B-B}: we obtain then the Monge-Kantorovich PDE
\begin{equation}\label{MK}
  \left\lbrace 
    \begin{array}{ll}
      -\div(\mu D_\mu u)=f & \mbox{in}\ \D'\\
      u\in Lip_1(\Omega) & \\
      |D_\mu u|=1 & \mu-\mbox{a.e.}
    \end{array}
  \right.
\end{equation}
In the system above the measure $\mu$ is called transport density and
plays a role in the transport problem and in some of its applications.

\begin{remark}
  The same conclusion holds in the more general framework of
  elasticity considered in \cite{B-B} in which the function $u$ is
  vector valued, a Dirichlet region $\Sigma$ is present, the bulk
  energy $\frac{1}{2}|\nabla u|^2$ is replaced by a convex
  $p$-homogeneous function $j(\nabla^{sym}u)$. For $f\in
  Lip_{1,\rho}(\Omega,\R^n)^\#$ the Monge-Kantorovich PDE (\ref{MK}) then takes
  the form
  \begin{equation}
    \left\lbrace 
      \begin{array}{ll}
        -\div(\sigma\mu)=f & \mbox{in}\ \D'(\Omega\setminus\Sigma)\\
        \sigma\in\partial j_\mu(x,e_\mu(u)) & \mu-\mbox{a.e. on }\ \Omega\\
        u\in Lip_{1,\rho}(\Omega,\Sigma) & \\
        j_\mu(x,e_\mu(u))=\frac{1}{p} & \mu-\mbox{a.e.}\\
        \mu(\Sigma)=0
      \end{array}
    \right.
  \end{equation}
  where we refer to \cite{B-B} for the precise meaning of
  $j_\mu,e_\mu,Lip_{1,\rho}$. The analogy with the quoted results
  remains also in the scalar case where the transportation problem can
  be written for $f\in\lips$. Note that in this case we cannot
  decompose $f$ into $f^+-f^-$ because $f$ is not in general a
  measure.
\end{remark}

%%%%%%%%%%%%%%%%%%%%%%%%%%%%%%%%%%%%%%%%%%%%%%%%%%%%%%%%%%%%%%%%%%%%%%%%%%%%
%%%%%%%%%%%%%%%%%%%%%%%%%%%%%%%%%%%%%%%%%%%%%%%%%%%%%%%%%%%%%%%%%%%%%%%%%%%%
\subsection{Duality, transport plan and transport densities}
In order to construct the analogous of  an optimal transport measure in the
general case of a distribution $f\in \xn$, we need to extend the Monge-Kantorovich 
duality and also to find a suitable
generalization of the concept of transport plan. To that aim we now present
a construction which will be shown to encompass 
the classical theory.

To each $\varphi\in\C^1_0(\R^N)$ we associate the function
$$D_\varphi(x,v,t):=\left\{\begin{array}{ll}
    \frac{\varphi(x+tv)-\varphi(x)}{t} & \mbox{if}\ t\neq0\\
    D\varphi(x)\cdot v & \mbox{if}\ t =0,
  \end{array}\right.$$
which belongs to $\C_0(\Omega\times
S^{N-1}\times[0,\infty))$. Then we can seek for a positive measure
$\sigma\in\M(\Omega\times S^{N-1}\times[0,\infty))$ which
minimizes the total variation among all the positive measures $\sigma$
such that
\begin{equation}\label{proietto}
  \int_{\Omega\times S^{N-1}\times[0,\infty)}D_\varphi(x,v,t)\,d\sigma
  =\langle f,\varphi\rangle\qquad\forall\varphi\in\C^1_0 (\R^N). 
\end{equation}
Then the new variational problem that we will consider reads as:
\begin{equation}\label{dualmin}
  \min\{\|\sigma\|\ :\ \sigma\in\M^+(\Omega\times S^{N-1}\times[0,\infty)),\ 
\sigma\ \mbox{satisfies }(\ref{proietto})\}.
\end{equation}

Before getting into the proof of existence in the general case of
$f\in \xn$ let us compare problem (\ref{dualmin}) with the
classical Monge-Kantorovich problem so that we can understand the
meaning of an optimal $\sigma$ in the classical case.  We will see
that $\sigma$ contains both optimal transport plans and optimal
transport densities.

%%%%%%%%%%%%%%%%%%%%%%%%%%%%%%%%%%%%%%%%%%%%%%%%%%%%%%%%%%%%%%%%%%%%%%%%%%%%
%%%%%%%%%%%%%%%%%%%%%%%%%%%%%%%%%%%%%%%%%%%%%%%%%%%%%%%%%%%%%%%%%%%%%%%%%%%%
\subsection{Comparisons with the classical case}
Let $f^+$ and $f^-$ be two finite and positive measures in $\Omega$ of
equal total mass and let $f=f^+-f^-$. Consider the map
$p:\Omega\times\Omega\to\Omega\times
S^{N-1}\times[0,\infty)$ defined by $$p(x,y):=\left\{\begin{array}{ll}
    (x,\frac{x-y}{|x-y|},|x-y|)& \mbox{if}\ x\neq y\\
    (x,e_1,0)& \mbox{if}\ x =y,
  \end{array}\right.$$
where the choice of $e_1$ is arbitrary and it is not relevant for what
follows.

\begin{prop}
  Let $f$ be a measure as above and let $\gamma$ be a transport plan
  for $f$; then the measure
  $p_\sharp(|x-y|\gamma)\in\M^+(\Omega\times S^{N-1}\times
  [0,\infty))$ satisfies the property (\ref{proietto}). Moreover if
  $\gamma$ is an optimal transport plan then $p_\sharp(|x-y|\gamma)$
  is optimal for problem (\ref{dualmin}). Therefore
  $$\W^1(f)=\inf_{\stackrel{\sigma\in\M^+(\Omega\times
      S^{N-1}\times[0,\infty))}{\pi_\sharp\sigma=f}}\|\sigma\|.$$
\end{prop}

\begin{proof}
  Let $\varphi\in\C^1_0(\R^N)$ then
  \begin{equation}\label{proxima}
    \begin{array}{ll}
      \ds\int_{\Omega\times S^{N-1}\times[0,\infty)} D_\varphi(x,v,t)\,dp_\sharp(|x-y|\gamma)&=\ds\int_{\Omega\times\Omega} D_\varphi(p(x,y))|x-y|\,d\gamma\\
      &=\ds\int_{\Omega\times\Omega}
      \varphi(x)-\varphi(y)\,d\gamma=\langle f,\varphi\rangle
    \end{array}
  \end{equation}
  which shows the first part of the claim.

  For the minimality: given a measure $\sigma$ which satisfies
  (\ref{proietto}) we have the inequality
  \begin{equation}\label{proxima2}
    \begin{array}{ll}
      \|\sigma\|&\ds\ge\sup_{\varphi\in\C^1_0(\R^N)\cap Lip_1(\Omega)}\int_{\Omega\times S^{N-1}\times[0,\infty)} D_\varphi(x,v,t)\,d\sigma\\
      &=\ds\sup_{\varphi\in\C^1_0(\R^N)\cap Lip_1(\Omega)}\langle f,\varphi\rangle=\int_{\clomega}|x-y|\,d\gamma.
    \end{array}
  \end{equation}
  On the other hand equation (\ref{proxima}) implies that
  \begin{equation}\label{proxima3}
    \begin{array}{ll}
      \ds\int_{\clomega}|x-y|\,d\gamma&=\ds\sup_{\varphi\in\C^1_0(\R^N)\cap Lip_1(\Omega)}\int_{\Omega\times S^{N-1}\times[0,\infty)} D_\varphi(x,v,t)\,dp_\sharp(|x-y|\gamma)\\
      &\le\ds\sup_{\psi\in\C_b(\clomega\times S^{N-1}\times[0,\infty)}\int_{\Omega\times S^{N-1}\times[0,\infty)}\psi(x,v,t)\,dp_\sharp(|x-y|\gamma)\\
      &\le\ds\int_{\clomega}|x-y|\,d\gamma
    \end{array}
  \end{equation}
  and the third term in the last inequality is the total variation of
  $p_\sharp(|x-y|\gamma)$.
\end{proof}

However, in the classical setting there is another way to build an
optimal $\sigma$. Indeed let $\nu$ be optimal for problem
(\ref{divcurr}) and consider the polar decomposition of $\nu$ as
$v|\nu|$ where $v$ is an unitary vector field. Define the map
$p_\nu:\spt\nu\to\clomega\times S^{N-1}\times[0,\infty)$ which to $x$
associates $(x,v(x),0)$.

\begin{prop} Let $f$ be a measure as above and let $\nu$ be optimal
  for problem (\ref{divcurr}).  Then the measure
  $(p_\nu)_\sharp\nu\in\M^+(\Omega\times
  S^{N-1}\times[0,\infty))$ is optimal for problem (\ref{dualmin}) and
  it is supported on $\Omega\times S^{N-1}\times \{0\}$.
\end{prop}

\begin{proof} First let us show that $\sigma=(p_\nu)_\sharp\nu$
  satisfies (\ref{proietto}). For every $\varphi\in\C^1_0(\R^N)$
  $$\int_{\clomega\times S^{N-1}\times[0,\infty)} D_\varphi(x,v,t) \,d\sigma
  =\int_\Omega D\varphi\,d\nu=\langle\varphi,f\rangle.$$
  As in
  (\ref{proxima2}) above the fact that $\sigma$ satisfies
  (\ref{proietto}) already implies the inequality $$\sup_{\varphi\in
    Lip_1(\Omega)}\langle\varphi,f\rangle\le\|\sigma\|.$$
  On the other
  hand by definition of $\sigma$ we have $\|\sigma\|=\|\nu\|$ and by
  the optimality of $\nu$ we obtain $$\sup_{\varphi\in
    Lip_1(\Omega)}\langle\varphi,f\rangle=\|\nu\|.$$
\end{proof}

%%%%%%%%%%%%%%%%%%%%%%%%%%%%%%%%%%%%%%%%%%%%%%%%%%%%%%%%%%%%%%%%%%%%%%%%%%%%
%%%%%%%%%%%%%%%%%%%%%%%%%%%%%%%%%%%%%%%%%%%%%%%%%%%%%%%%%%%%%%%%%%%%%%%%%%%%
\subsection{Existence and structure of minimizers for problem
  (\ref{dualmin})}
Let us now prove the existence of an optimal $\sigma$ for general
$f\in \xn$.

\begin{theorem}\label{exist&dual}
  Let $f\in\lipp$. Then there exists an optimal measure $\sigma$ for
  the transportation problem (\ref{dualmin}). Moreover the minimal
  value for problem (\ref{dualmin}) coincides with the supremal value
  for problem (\ref{newflat}).
\end{theorem}

\begin{proof}
  Again for each $\varphi\in\C^1_0(\R^N)$ we define $D_\varphi$ as
  above and we consider $Y:=\{D_\varphi\ :\ \varphi\in\C^1_0(\R^N)\}$
  equipped with the $L^\infty$ norm.

  Define the linear functional $F:Y\to\R$ by $F(D_\varphi):=\langle
  f,\varphi\rangle$ and notice that
  $$|F(D_\varphi)|\le\|f\|_{\lipp}\|\varphi\|_{\lip}=\|f\|_{\lipp}\|D_\varphi\|_\infty\qquad\forall
  D_\varphi\in Y.$$
  The Hahn-Banach theorem then provides an extension
  $\tilde F$ of $F$ to the space of bounded and continuous functions
  on $\Omega\times S^{N-1}\times[0,\infty)$ which preserves
  the norm of $F$, and such extension is represented by a measure
  $\sigma\in\M(\Omega\times S^{N-1}\times[0,\infty) )$
  which verifies
  \begin{enumerate}
  \item $F(D_\varphi)=\int_{\Omega\times
      S^{N-1}\times[0,\infty)}D_\varphi\,d\sigma$ for all $\varphi \in
    \C^1_0(\R^N)$;
  \item $\|\sigma\|=\|\tilde F\|=\|F\|$.
  \end{enumerate}
  In particular the first of the previous conditions implies that
  $\sigma$ satisfies (\ref{proietto}) and this gives $$\langle
  f,\varphi\rangle=\int_{\Omega\times
    S^{N-1}\times[0,\infty)}D_\varphi\,d\sigma\le\|\sigma\|\qquad\forall\varphi\in
  \C^1_0(\R^N).$$
  Then the equality $$\|\sigma\|=\|\tilde{F}\|=\|F\|$$
  implies both the equality
  \begin{equation}\label{eqnorm}
    \|\sigma\|=\sup\{\langle f,\varphi\rangle\ :\ \varphi\in
    \C^1_0(\R^N)\cap Lip_1(\Omega)\}
  \end{equation}
  and the minimality of $\sigma$.
\end{proof}

When $f$ is a measure, in \cite{B-B} it is shown that an optimal
transportation density can be obtained through an optimal plan
$\gamma$ considering the total variation of the measure $\nu$ defined
by the formula
\begin{equation}\label{formula}
  \langle\nu,\varphi\rangle=\int\Big(\int_{S_{x,y}}\varphi\,d\H^1\Big)\gamma(dx,dy)
\end{equation}
where $S_{x,y}$ denotes the segment joining $x$ to $y$ (a geodesic
line in the general case).  We show that the same can be done when
$f\in \xn$. Decompose an optimal $\sigma$ in the sum of two
parts: $$\sigma_0:=\sigma\res(\Omega\times
S^{N-1}\times\{0\}),\qquad\sigma_+=\sigma-\sigma_0,$$
and define the
map $\pi:\Omega\times
S^{N-1}\times(0,+\infty)\to\Omega\times\R^N$ as
$$\pi(x,v,t)=(x,x+tv).$$
Using a notation which is reminiscent of
transport plans we define $\gamma_+:=\pi_\sharp\sigma_+,$ and then in
correspondence with $\sigma_+$ we consider the measures $\nu_+$
defined by:
\begin{equation}\label{nuerre}
  \langle\psi,\nu_+\rangle:=\int_{\clomega\times\clomega}
  \frac{1}{|x-y|}\int\frac{y-x}{|y-x|}\cdot\psi(\cdot)\,d\H^1\res[x,y]\,d\gamma_+
  (x,y). 
\end{equation}
To $\sigma_0$ instead we associate $\nu_0$ defined by
\begin{equation}\label{nu0}
  \langle\psi,\nu_0\rangle:=\int_{\Omega\times S^{N-1}}\psi(x)\cdot V\,d\sigma_0(x,V).
\end{equation}

\begin{theorem}\label{divergenze}
  Let $\nu_0$ and $\nu_+$ be defined by (\ref{nu0}) and
  (\ref{nuerre}). Then
  \begin{equation}
    -\div(\nu_0+\nu_+)=f.
  \end{equation}
  Moreover if $\sigma$ is optimal then $\nu=\nu_0+\nu_+$ is also
  optimal for (\ref{divcurr}).
\end{theorem}

\begin{proof}
  For every $\varphi\in\C^1_0(\Omega)$ one has
  \begin{eqnarray*}
    \langle-\div(\nu_0+\nu_+),\varphi\rangle&=& \int_{\clomega\times\clomega}\frac{1}{|y-x|}\int _0^1\nabla(x+t(y-x))\cdot(y-x)\,dt\,d\tilde{\gamma}(x,y)\\
    & & +\int_{\clomega\times S^{N-1}}\nabla\varphi(x)\cdot V\,d\sigma_0(x,V,0)\\
    &=& \int_{\clomega\times\clomega}\frac{\varphi(y)-\varphi(x)}{|y-x|}\,d\sigma_+(x,y)\\
    & & +\int_{\clomega\times S^{N-1}}\nabla\varphi(x)\cdot V\,d\sigma_0 (x,V,0)\\
    &=& \int_{\clomega \times S^{N-1}\times[0,\infty)} D_\varphi(x,v,t)\,d\sigma=\langle f,\varphi\rangle.
  \end{eqnarray*}
  About the minimality first observe that directly from the formula above 
  one obtains an estimate on the total variation of $\nu$:
  \begin{equation}\label{stimaup}
    \|\nu\|\le\|\sigma\|.
  \end{equation}
  On the other hand
  \begin{equation}\label{numero}
    \begin{array}{ll}
      |\nu|&=\ds\sup_{\psi\in \C_0(\Omega),\,\|\psi\|_\infty\le1}\langle
      \nu,\psi\rangle\geq\sup_{\varphi\in \C^1(\Omega)\cap Lip_1}\langle
      \nu,\nabla\varphi\rangle\\
      &=\ds\sup_{\varphi\in \C^1(\Omega)\cap Lip_1}\langle f,\varphi\rangle=\|f\|_{\lipp}
    \end{array}
  \end{equation}
  and if $\sigma$ is optimal $\|f\|_{\lipp}=\|\sigma\| $ thus giving
  equality in (\ref{stimaup}) and the minimality of $\nu$.
\end{proof}

\begin{remark}\label{optimalnu} In particular Theorem \ref{divergenze} allows us to prove the
  formula $$\W^1(f)=\min\{\|\nu\|\ :\ \nu\in\M(\Omega,\R^N),\ -\div\nu
  =f\}.$$
\end{remark}

%%%%%%%%%%%%%%%%%%%%%%%%%%%%%%%%%%%%%%%%%%%%%%%%%%%%%%%%%%%%%%%%%%%%%%%%%%%%
%%%%%%%%%%%%%%%%%%%%%%%%%%%%%%%%%%%%%%%%%%%%%%%%%%%%%%%%%%%%%%%%%%%%%%%%%%%%
\section{A decomposition of $\xn$ and the distance to $\lips$}

We will now apply the theory constructed so far to   
give an ``orthogonal decomposition'' of $\xn$ and
to compute the distance of a distribution
$f\in\lipp$ to the space $\lips$ in terms of the problems introduced
in the previous sections. Let us recall that, as remarked in the
introduction, the space $\lips$ is a closed subspace of
$\xn$ and contains the weak Jacobians of maps in
certains Sobolev spaces.

Let $f\in\lipp$; then by Theorem \ref{divergenze} $f$ may be written
as $f=-\div\nu$ for a suitable vectorial measure. Recalling definition \ref{tangmeas} 
we can further decompose $\nu$ as
$$\nu=\nu_T+\nu_N$$
where the measure
$\nu_T\in \tang$ is a tangent measure and $\nu_N\in \norm$ is a normal measure. 
In other words we may write  $\nu_T=v_T |\nu|$ and $\nu_N=v_N |\nu|$ where $v_T (x)\in
{\mathcal T}_{|\nu|}(x)$  and $v_N(x)\in{\mathcal N}_{|\nu|}(x)$
for $|\nu|$-a.e. $x$.
We will use the following technical result:

\begin{lemma}\label{relaxandgo} Let $\alpha$ be a positive Radon measure in $\R^N$ 
and let $\eta \in (L^1_\alpha)^N $.
Then there exists a sequence $\{\varphi_n\} \subset \C^1_0 (\R^N)$  
such that: $\varphi_n \to 0$ uniformly,
  $|\nabla \varphi_n|\leq 1$ and 
$$ \lim_{n \to \infty } \int \nabla \varphi_n (x) \cdot \eta(x) d \alpha = \int |\eta_N
(x)| d \alpha,$$
where $\eta_N(x) \in (T_\alpha(x))^\bot$ denotes the normal component of $\eta(x)$.
\end{lemma}
\begin{proof} Consider the integrand
$$j (x,z)= \eta(x) \cdot z +\chi_{\{|z| \leq 1\}}, $$
and the functional
\begin{equation}
F(\varphi)= \left\{\begin{array}{ll}
\int j(x, \nabla \varphi) d \alpha & \mbox{if} \ \varphi \in \C^1_0 (\R^N), \\
+\infty & \mbox{otherwise}.
\end{array}
\right.
\end{equation}
Denote by $\overline{F}: \C_0 (\R^N) \to \R$ the relaxed functional of
$F$ with respect to the uniform convergence, then we claim that 
\begin{equation}\label{relaxineq}
\overline{F}(0) = - \int |\eta_N(x)| d \alpha.
\end{equation}
By definition of relaxed functional (\ref{relaxineq}) implies that
there exists a sequence $\{\psi_n\} \subset \C^1_0 (\R^N)$  such that: $\psi_n \to 0$ uniformly,
  $|\nabla \psi_n|\leq 1$ and $\lim_{n \to \infty}  \int \nabla \psi_n
  (x) \cdot \eta(x) d \alpha = - \int |\eta_N(x)| d \alpha$. Then
  it is enough to consider $\varphi_n:=-\psi_n$ to obtain the
  conclusion of the lemma. 
Let us then prove  (\ref{relaxineq}). By convexity 
$ \overline{F}(0)= F^{**} (0)$ and by definition  $F^{**} (0) =\sup_{g \in \M (\Omega)}
{-F^* (g)}=- \inf_{g \in \M (\Omega)} F^* (g)$. 
We compute now $F^*$ .We notice that $F= J\circ A$ where $J$ denotes the integral 
functional $J: p\in \C_0(\R^N;\R^N) \mapsto \int j(x,p) \, d \alpha$
and $A: u\in \C^1_0(\R^N) \mapsto \nabla u\in \C_0(\R^N;\R^N)$.
As $J$ is convex continuous at $p=0$, by a classical duality result 
(see for instance \cite{ency}), we have 
$$F^* (g)=\inf \{J^*(\sigma)\ | \ -\div \ \sigma =g\},  $$
where $J^*$ is the Fenchel conjugate of $J$ on the dual space $\M(\R^N;\R^N)$.
A simple computation shows that $j^*(x,w)=|w-\eta(x)|$ and
by applying \cite{Bou-Val}, we have
$$J^*(\sigma)= \int j^* (\frac{d \sigma}{d \alpha})d \alpha + \int h(x,\sigma_s)$$
where $\sigma_s$ represent the singular part of $\sigma$ with respect to $\alpha$ and
$$h(x,z)=\sup \{\psi(x)\cdot z \ | \ \int j (x,\psi(x))d \alpha < \infty, \ \psi\in   \C_0(\R^N;\R^N)  \ ,\ |\psi|\leq 1 \ \} \ =\ |z|. $$
Therefore if we decompose all measures $\sigma$ such that $-\div \sigma=g$ in 
its absolutely continuous  and singular parts with respect to $\alpha$ so that $\sigma= w \alpha + \sigma _s$,
we can write
$$F^* (g)=\inf \{\int_{\R^N} |w-\eta|d \alpha+\int_{spt \alpha}|\sigma_s |\ | \ 
-\div (w \alpha+\sigma_s) =g\}. $$
Let us choose $w=\eta_T$, $\sigma_s=0$. Then $\overline{g}=-\div (\eta_T \alpha)$
and we get
$$ \inf {F^*(g)} \leq F^*(\overline{g})= \int |\eta_N| d \alpha, $$
and this prove the first inequality of (\ref{relaxineq}).
To prove the opposite inequality for a given $g= -\div (w \alpha + \sigma _s)$
define  $m=\alpha +\sigma_s$ and set
\begin{equation*}
q(x)= \left\{ 
\begin{array}{ll}
w(x) & \alpha-a.e.\\
\frac{d \sigma_s}{d |\sigma_s|} & \sigma_s-a.e..
\end{array}
\right.
\end{equation*}
Since $g=-\div\ (q(x) m)$ is a measure, by  Proposition \ref{divtang}, there holds $q(x) \in
T_m(x)$ for $m$-a.e. $x$ and then by Lemma \ref{submeas}\  $w \in T_\alpha(x)$
for $\alpha$-a.e. $x$. Thus
$$ \int_{\R^N} |w-\eta|d \alpha+\int_{spt \alpha}|\sigma_s |=
\int_{\R^N}( |w-\eta_T|+|\eta_N|)d \alpha+\int_{spt \alpha}|\sigma_s |
\geq  \int_{\R^N}|\eta_N|d \alpha. $$
It follows that $\inf F^* \ge \int_{\R^N}|\eta_N|d \alpha $ and we are led
to the equality in  (\ref{relaxineq}). 

\end{proof}

We are now in position to state the  main theorem of this section.

\begin{theorem}\label{projector} For every $f\in\lipp$, there holds 
$$\W^1(f,\lips)=\min\left\{\int|\nu_N|\ :\ \nu\in\M(\Omega,\R^N),\
  -\div\nu=f\right\}.$$
Moreover there exists a unique decomposition \
$ f=f_T+f_N $
with $f_T \in \x0s$ and $f_N= \div \ \beta$ for some normal measure $\beta\in \norm $. We have in addition 
$$ \W^1 (f)= \W^1(  f_T)+ \W^1 (f_N).$$
\end{theorem}

\begin{proof}  By Theorem \ref{divergenze}, there exists  
a measure $\nu$ such that $-\div\ \nu=f$. By the definition of $\W^1(f,\lips)$ and recalling that elements of $\x0s$ can be represented as divergence of
tangential measures (see \ref{divtang}), we derive successively
\begin{eqnarray}
\W^1(f,\lips) &=& \inf_{g \in \x0s} \sup_{ u \in \C^1_0 \cap Lip_1}
\langle f-g , u\rangle=\\
&=&  \inf_{g \in \x0s} \sup_{ u \in \C^1_0 \cap Lip_1} \langle -\div
\nu_N -\div \nu_T -g , u\rangle=\\
&=&  \inf_{G \in \tang} \sup_{ u \in \C^1_0 \cap Lip_1} \langle -\div
\nu_N -\div  G , u\rangle=\\
&=&  \inf_{G \in \tang} \sup_{ u \in \C^1_0 \cap Lip_1} \langle 
\nu_N + G , \nabla u\rangle\leq \int|\nu_N|\ . 
\end{eqnarray}

On the other hand, by applying Lemma \ref{relaxandgo} to the measure 
$\nu_N + G $ of the last inequality, we obtain an equality. It follows
in particular that for all $\nu$ such that $-\div \ \nu=f$ 
\begin{equation} 
\W^1(f,\lips)= \int |\nu_N| .
\end{equation}
The decomposition $f= f_T+f_N$ of an element $f\in \xn$,
is obtained by considering any $\nu$ such
that $-\div \ \nu=f$ and $ \W^1 (f)= \int |\nu|$ (see Remark \ref{optimalnu})
and  then by setting: $f_T:= - \div \nu_T$ and $f_N= -\div \nu_N$.
The uniqueness of such decomposition is straightforward since 
the divergence of a normal measure cannot belong to $\x0s$ unless
it vanishes.
\end{proof}

A second formula is related to the measures
$\sigma\in\M^+(\Omega\times S^{N-1}\times[0,\infty))$ such that
$\pi_\sharp\sigma=f$ which are then admissible for problem
(\ref{dualmin}). Indeed we introduced the natural decomposition
$\sigma=\sigma_0+\sigma_+$ and by equations (\ref{nu0}) and
(\ref{nuerre}) we associated a measure $\nu_0$ to $\sigma_0$ and a
measure $\nu_+$ to $\sigma_+$. By construction $\nu_+$ is always a
tangential measure while $\nu_0$ is not necessarily so.

\begin{theorem} For every $f\in\lipp$ we have
  $$\W^1(f,\lips)=\inf \{\|\sigma_0\|\ :\
  \sigma\in\M^+(\Omega\times S^{N-1}\times[0,\infty))\
  \mbox{and}\ \pi_\sharp\sigma=f\}.$$
\end{theorem}

\begin{proof} Let $\sigma$ be such that $\pi_\sharp\sigma=f$ then as
 noticed before the measure $\nu_+$ associated to $\sigma$ by equation
  (\ref{nuerre}) is always tangential and then:
  \begin{equation}\label{estimate}
\W^1(f,\lips)\le\W^1(f,-\div\nu_+)\le\|\nu_0\|\le\|\sigma_0\|.
\end{equation}

  Let $f_n$ be a sequence of measures in $\M (\Omega)$ such that
  $f_n=f^+_n-f^-_n$ with $\|f_n^+\|=\|f_n^-\|<\infty$ and
  $$\W^1(f,f_n)\le\W^1 (f,\lips)+\varepsilon_n.$$
  Let
  $\xi^n\in\M(\Omega\times S^{N-1}\times[0,\infty))$ of
  minimal total variation among the positive measures such that
  $\pi_\sharp\sigma=f-f_n$. We decompose $\xi^n$ as
  $\xi^n_0+\xi^n_+$. Then $\|\xi^n\| = \|\xi^n_0\|+\|\xi^n_+\|$  and therefore  \begin{equation}\label{numero2}
      \|\xi^n_+\| \ \le\ \W^1(f,\lips)+\varepsilon_n-\|\xi^n_0\|\
      \le\ \W^1(f, \lips)+\varepsilon_n.  \end{equation}
  Let $\gamma_n$ be an optimal transport plan for $f_n$ and consider
  $\sigma^n:=\xi^n + p_\sharp(|x-y|\gamma_n)$ where $p$ is
  the map introduced in Subsection 3.3. By the linearity
  $\pi_\sharp\sigma^n=f$ that is $\sigma^n$ is admissible
  and by construction $\sigma^n_0=\xi^n_0$. Then (\ref{numero2})
 shows that $\sigma^n$ is optimal up to infinitesimal constant $\varepsilon_n$.
\end{proof}

\bigskip\bigskip

\textbf{Acknowledgments.} The research of the second and third
authors is part of the project {\it``Metodi variazionali nella teoria
  del trasporto ottimo di massa e nella teoria geometrica della
  misura''} of the program PRIN 2006 of the Italian Ministry of the
University.

%%%%%%%%%%%%%%%%%%%%%%%%%%%%%%%%%%%%%%%%%%%%%%%%%%%%%%%%%%%%%%%%%%%%%%%%%%%%%%

\end{document}